\documentclass[12pt]{amsart}
\title{A remark about 6j symbols and young semi-normal form}
\author{Daniel Barter}
\date{\today}

\makeatletter
\def\@endtheorem{\endtrivlist}% NEW
\makeatother % gets rid of line breaks after theorems

\usepackage{easyfig} %nice package which builds on graphicx
\usepackage{float}
\usepackage{booktabs} %fancy tables
\usepackage{amsmath,amssymb}
\usepackage{tikz-cd}
\usepackage{mathrsfs} % Sheafy Fonts
\usepackage{tikz}
\usetikzlibrary{calc}

\usepackage[boxsize = 0.5em]{ytableau} %% \ydiagram{2,2}
\usepackage{enumerate} %allows the enumeration symbols to be changed
\usepackage[margin=2.5cm]{geometry} % set margins

\theoremstyle{plain}
\newtheorem{Theorem}{Theorem}

\newtheorem{Proposition}[Theorem]{Proposition}

\theoremstyle{definition}
\newtheorem{Definition}[Theorem]{Definition}

\DeclareMathOperator{\End}{End}

\DeclareMathOperator{\GL}{GL}
\DeclareMathOperator{\SL}{SL}

\newcommand{\Vect}{{\bf Vec}}
\newcommand{\XX}{{\bf X}}

\begin{document}

\maketitle

\section{Introduction} \label{sec:introduction}

A semi-simple tensor category is determined up to equivalence by its Grothendieck ring and its $6j$ symbols with respect to a set of tree basis vectors. The $6j$ symbols are the coordinate representation of the associator. Despite their importance, we only know explicit formulas for $6j$ symbols in a few special cases. In the $\SL_2(\mathbb{C})$ and $U_q(\mathfrak{sl}_2)$ cases, explicit formulas for the $6j$-symbols are computed in \cite{MR1366832}. As far as the author is aware, the only other case where explicit $6j$-symbols are known is $G$-graded vector spaces for $G$ a finite group. In this case, associators are cohomology classes in $H^3(G,\mathbb{C}^{\times})$. 

In this note, we compute a large number of $6j$ symbols inside the tensor category consisting of polynomial representations of $\GL(\infty)$. This tensor category is studied in detail by Sam and Snowden in \cite{MR3430359,MR3376738,SS2012_2}. More precisely, we have:
\begin{Theorem} \label{thm:main_theorem}
Let $\lambda \subseteq \mu$ be partitions such that $\mu \backslash \lambda$ has two boxes not contained in a single row or column. Then
\[
(j_{\mu}^{\lambda,\ydiagram{1},\ydiagram{1}})^{-1} = 
\begin{pmatrix}
\frac{k}{k+1} & \frac{-k}{k-1} \\
1 & 1
\end{pmatrix}.
\]
Where $k$ is the axial distance in the skew partition $\mu \backslash \lambda$. These $6j$ symbols are uniquely defined up to scaling.
\end{Theorem}

We describe computations in tensor categories using string diagrams. A good introduction to string diagrams is \cite{MR2767048}. For a complete and rigorous introduction to tensor categories, we direct the reader to \cite{MR3242743}.

The author would like to thank Corey Jones and Scott Morrison for many useful conversations and their support.
\section{What are 6j symbols?} \label{sec:what_are_6j_symbols?}

\begin{Definition} \label{def:tree_string_diagrams}
 Let $\XX$ be a semi-simple tensor category. Index the simple objects with a set $\Lambda$. Choose a basis for each $\XX(\mu,\lambda \otimes \nu)$ denoted by
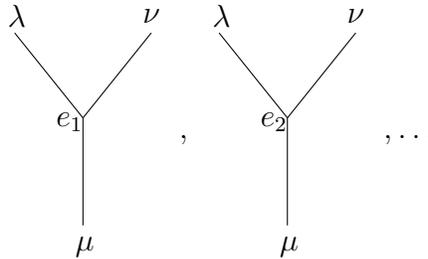
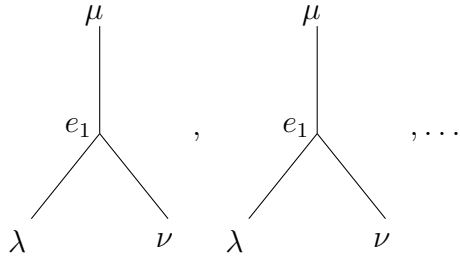
\[
\begin{tikzpicture}[yscale=-1,scale=0.015,baseline={([yshift=-.5ex]current bounding box.center)}]
\begin{scope}[shift={(0.00mm,719.29mm)}]
% path id='path4140'
% path spec='M 887.14286,53.076506 281.42857,-701.2092'
\draw [fill=none,draw=black] (887.14mm,53.08mm)
-- (281.43mm,-701.21mm)
;
% path id='path4142'
% path spec='M 886.64285,53.076506 1492.3571,-701.2092'
\draw [fill=none,draw=black] (886.64mm,53.08mm)
-- (1492.36mm,-701.21mm)
;
% path id='path4144'
% path spec='m 886.79138,52.412731 0,953.584069'
\draw [fill=none,draw=black] (886.79mm,52.41mm)
-- ++(0.00mm,953.58mm)
;
\node [black] at (302.86mm,-850mm) { $\lambda$ };
\node [black] at (1494.29mm,-850mm) { $\nu$ };
\node [black] at (902.86mm,1200mm) { $\mu$ };
\node [black] at (768.57mm,86.65mm) { $e_1$ };
\end{scope}
\end{tikzpicture} \; , \;
\begin{tikzpicture}[yscale=-1,scale=0.015,baseline={([yshift=-.5ex]current bounding box.center)}]
\begin{scope}[shift={(0.00mm,719.29mm)}]
% path id='path4140'
% path spec='M 887.14286,53.076506 281.42857,-701.2092'
\draw [fill=none,draw=black] (887.14mm,53.08mm)
-- (281.43mm,-701.21mm)
;
% path id='path4142'
% path spec='M 886.64285,53.076506 1492.3571,-701.2092'
\draw [fill=none,draw=black] (886.64mm,53.08mm)
-- (1492.36mm,-701.21mm)
;
% path id='path4144'
% path spec='m 886.79138,52.412731 0,953.584069'
\draw [fill=none,draw=black] (886.79mm,52.41mm)
-- ++(0.00mm,953.58mm)
;
\node [black] at (302.86mm,-850mm) { $\lambda$ };
\node [black] at (1494.29mm,-850mm) { $\nu$ };
\node [black] at (902.86mm,1200mm) { $\mu$ };
\node [black] at (768.57mm,86.65mm) { $e_2$ };
\end{scope}
\end{tikzpicture} \; , \dots
\]
\noindent and let 
\[
\begin{tikzpicture}[yscale=-1,scale=0.015,baseline={([yshift=-.5ex]current bounding box.center)}]
\begin{scope}[shift={(0.00mm,719.29mm)}]
% path id='path4140'
% path spec='M 888.57143,251.04092 282.85714,1005.3267'
\draw [fill=none,draw=black] (888.57mm,251.04mm)
-- (282.86mm,1005.33mm)
;
% path id='path4142'
% path spec='M 888.78571,250.68377 1494.4999,1004.9695'
\draw [fill=none,draw=black] (888.79mm,250.68mm)
-- (1494.50mm,1004.97mm)
;
% path id='path4144'
% path spec='m 888.93424,251.34755 0,-953.58407'
\draw [fill=none,draw=black] (888.93mm,251.35mm)
-- ++(0.00mm,-953.58mm)
;
\node [black] at (157.14mm,1200.93mm) { $\lambda$ };
\node [black] at (1462.86mm,1200mm) { $\nu$ };
\node [black] at (840.00mm,-800.78mm) { $\mu$ };
\node [black] at (700mm,200mm) { $e_1$ };
\end{scope}
\end{tikzpicture} \; , \;
\begin{tikzpicture}[yscale=-1,scale=0.015,baseline={([yshift=-.5ex]current bounding box.center)}]
\begin{scope}[shift={(0.00mm,719.29mm)}]
% path id='path4140'
% path spec='M 888.57143,251.04092 282.85714,1005.3267'
\draw [fill=none,draw=black] (888.57mm,251.04mm)
-- (282.86mm,1005.33mm)
;
% path id='path4142'
% path spec='M 888.78571,250.68377 1494.4999,1004.9695'
\draw [fill=none,draw=black] (888.79mm,250.68mm)
-- (1494.50mm,1004.97mm)
;
% path id='path4144'
% path spec='m 888.93424,251.34755 0,-953.58407'
\draw [fill=none,draw=black] (888.93mm,251.35mm)
-- ++(0.00mm,-953.58mm)
;
\node [black] at (157.14mm,1200.93mm) { $\lambda$ };
\node [black] at (1462.86mm,1200mm) { $\nu$ };
\node [black] at (840.00mm,-800.78mm) { $\mu$ };
\node [black] at (700mm,200mm) { $e_1$ };
\end{scope}
\end{tikzpicture} \; , \dots
\]
\noindent be the dual basis of $\XX(\lambda \otimes \nu,\mu)$. We call these diagrams {\bf tree string diagrams}. It is important to notice that the tree string diagrams are not canonically defined. 
\end{Definition}

\begin{Definition} \label{def:fusion_graph}
Pick a distinguished simple object $X \in \XX$. The {\bf fusion graph} of $X$ has vertices $\Lambda$ and the edges from $\lambda$ to $\mu$ are the distinguished basis vectors in $\XX(\mu,\lambda \otimes X)$. 
\end{Definition}

\begin{Proposition} \label{prop:hom_space_tree_basis}
Fix $\lambda \in \Lambda$. Then $\XX(\lambda,X^{\otimes n})$ has dimension the number of length $n$ paths from the tensor unit to $\lambda$ in the fusion graph for $X$. Moreover, an. explicit basis is given by string diagrams of the form
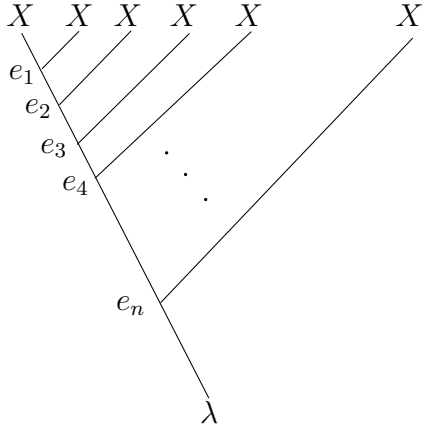
\[
\begin{tikzpicture}[yscale=-1,scale=0.03,baseline={([yshift=-.5ex]current bounding box.center)}]
\begin{scope}[shift={(0.00mm,719.29mm)}]
\draw [fill=black,draw=black] (659.15mm,-50.45mm) circle (4.80mm) ;
\draw [fill=black,draw=black] (743.15mm,45.55mm) circle (4.80mm) ;
\draw [fill=black,draw=black] (832.15mm,156.55mm) circle (4.80mm) ;
% path id='path4167'
% path spec='m 270.72496,-589.74787 -164.04877,166.8772'
\draw [fill=none,draw=black] (270.72mm,-589.75mm)
-- ++(-164.05mm,166.88mm)
;
% path id='path4171'
% path spec='m 502,-591.63777 -318,328'
\draw [fill=none,draw=black] (502.00mm,-591.64mm)
-- ++(-318.00mm,328.00mm)
;
% path id='path4173'
% path spec='M 16.656872,-581.15247 846.26401,1037.6598'
\draw [fill=none,draw=black] (16.66mm,-581.15mm)
-- (846.26mm,1037.66mm)
;
% path id='path4175'
% path spec='m 262.23968,-87.49391 497.80318,-492.14632 0,0 0,0'
\draw [fill=none,draw=black] (262.24mm,-87.49mm)
-- ++(497.80mm,-492.15mm)
-- ++(0.00mm,0.00mm)
-- ++(0.00mm,0.00mm)
;
% path id='path4177'
% path spec='m 341.43564,62.41273 692.96466,-650.53824 0,0'
\draw [fill=none,draw=black] (341.44mm,62.41mm)
-- ++(692.96mm,-650.54mm)
-- ++(0.00mm,0.00mm)
;
% path id='path4179'
% path spec='M 632.76363,613.95606 1749.9923,-559.84123'
\draw [fill=none,draw=black] (632.76mm,613.96mm)
-- (1749.99mm,-559.84mm)
;
\node [black] at (20.00mm,-395.64mm) { $e_1$ };
\node [black] at (90.00mm,-247.64mm) { $e_2$ };
\node [black] at (160.00mm,-71.64mm) { $e_3$ };
\node [black] at (256.00mm,96.36mm) { $e_4$ };
\node [black] at (500.00mm,628.36mm) { $e_n$ };
\node [black] at (12.00mm,-660mm) { $X$ };
\node [black] at (272.00mm,-660.64mm) { $X$ };
\node [black] at (488.00mm,-660.64mm) { $X$ };
\node [black] at (736.00mm,-660.64mm) { $X$ };
\node [black] at (1024.00mm,-660.64mm) { $X$ };
\node [black] at (1736.00mm,-660.64mm) { $X$ };
\node [black] at (848.00mm,1100mm) { $\lambda$ };
\end{scope}
\end{tikzpicture}
\]
\noindent we call such string diagrams {\bf tree basis vectors}
\end{Proposition}
\begin{proof}
Decompose $X^{\otimes n}$ using the fusion graph for $X$.
\end{proof}
\begin{Definition} \label{def:matrix_units}
Since $\XX$ is semi-simple, the Artin-Wedderburn theorem implies that $\End(X^{\otimes n})$ is a product of matrix algebras. Proposition \ref{prop:hom_space_tree_basis} implies that in the tree string basis, the matrix units in  $\End(X^{\otimes n})$ look like
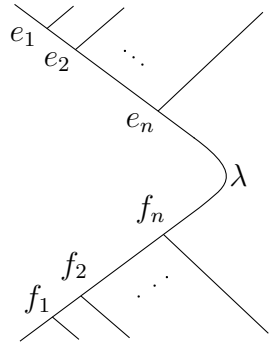
\[
\begin{tikzpicture}[yscale=-1,scale=0.03,baseline={([yshift=-.5ex]current bounding box.center)}]
\begin{scope}[shift={(0.00mm,719.29mm)}]
% path id='path4140'
% path spec='m 250,-500.49492 118.57143,-100 0,0'
\draw [fill=none,draw=black] (250.00mm,-500.49mm)
-- ++(118.57mm,-100.00mm)
-- ++(0.00mm,0.00mm)
;
% path id='path4142'
% path spec='M 377.14286,-407.63778 592.85714,-593.35206'
\draw [fill=none,draw=black] (377.14mm,-407.64mm)
-- (592.86mm,-593.35mm)
;
% path id='path4144'
% path spec='m 743.11513,-134.98471 464.67027,-438.4062'
\draw [fill=none,draw=black] (743.12mm,-134.98mm)
-- ++(464.67mm,-438.41mm)
;
% path id='path4146'
% path spec='m 107.07617,-608.32856 798.02051,593.9697'
\draw [fill=none,draw=black] (107.08mm,-608.33mm)
-- ++(798.02mm,593.97mm)
;
% path id='path4148'
% path spec='m 274,777.80752 118.57143,100.00001 0,0'
\draw [fill=none,draw=black] (274.00mm,777.81mm)
-- ++(118.57mm,100.00mm)
-- ++(0.00mm,0.00mm)
;
% path id='path4150'
% path spec='m 401.14286,684.95037 215.71428,185.7143'
\draw [fill=none,draw=black] (401.14mm,684.95mm)
-- ++(215.71mm,185.71mm)
;
% path id='path4152'
% path spec='m 767.11513,412.29732 464.67027,438.4062'
\draw [fill=none,draw=black] (767.12mm,412.30mm)
-- ++(464.67mm,438.41mm)
;
% path id='path4154'
% path spec='M 131.07617,885.64117 929.09668,291.67146'
\draw [fill=none,draw=black] (131.08mm,885.64mm)
-- (929.10mm,291.67mm)
;
% path id='path4156'
% path spec='M 901.42857,-17.046045 C 1141.8725,158.38478 1031.2563,204.80887 928.99729,291.64418'
\draw [fill=none,draw=black] (901.43mm,-17.05mm)
%%%% Warning: check controls
.. controls (1141.87mm,158.38mm) and (1031.26mm,204.81mm) .. (929.00mm,291.64mm)
;
\node [black] at (1100mm,150mm) { $\lambda$ };
\node [black] at (145.71mm,-464.78mm) { $e_1$ };
\node [black] at (297.14mm,-353.35mm) { $e_2$ };
\node [black] at (665.71mm,-84.78mm) { $e_n$ };
\node [black] at (208.34mm,700.86mm) { $f_1$ };
\node [black] at (370.78mm,550.34mm) { $f_2$ };
\node [black] at (709.90mm,300.67mm) { $f_n$ };
\draw [fill=black,draw=black] (596.18mm,-401.93mm) ellipse (2.65mm and 2.53mm) ;
\draw [fill=black,draw=black] (634.47mm,-373.36mm) ellipse (2.65mm and 2.53mm) ;
\draw [fill=black,draw=black] (673.76mm,-341.21mm) ellipse (2.65mm and 2.53mm) ;
\draw [fill=black,draw=black] (657.13mm,687.36mm) ellipse (2.65mm and 2.53mm) ;
\draw [fill=black,draw=black] (727.74mm,639.16mm) ellipse (2.65mm and 2.53mm) ;
\draw [fill=black,draw=black] (782.18mm,601.60mm) ellipse (2.65mm and 2.53mm) ;
\end{scope}
\end{tikzpicture}
\]
\noindent Equivalently, the irreducible representations of $\End(X^{\otimes n})$ are parameterized by the simple objects in $\XX$ which have a length $n$ path from the tensor unit in the fusion graph for $X$. The string diagrams defined in proposition \ref{prop:hom_space_tree_basis} form a basis for the corresponding representation.
\end{Definition}

\begin{Definition} \label{def:6j_symbols}
Fix $\lambda_1,\lambda_2,\lambda_3,\mu \in \Lambda$. Then we have two bases for $\XX(\mu,\lambda_1 \otimes \lambda_2 \otimes \lambda_3)$:

\[
\left\{ 
\begin{tikzpicture}[yscale=-1,scale=0.03,baseline={([yshift=-.5ex]current bounding box.center)}]
\begin{scope}[shift={(0.00mm,719.29mm)}]
% path id='path4136'
% path spec='M 125.71429,-467.63778 1457.1429,838.07656'
\draw [fill=none,draw=black] (125.71mm,-467.64mm)
-- (1457.14mm,838.08mm)
;
% path id='path4138'
% path spec='m 1048.5714,435.21936 462.8572,-905.71428 0,0'
\draw [fill=none,draw=black] (1048.57mm,435.22mm)
-- ++(462.86mm,-905.71mm)
-- ++(0.00mm,0.00mm)
;
% path id='path4140'
% path spec='m 625.71429,20.93365 308.57142,-500 0,0 0,0'
\draw [fill=none,draw=black] (625.71mm,20.93mm)
-- ++(308.57mm,-500.00mm)
-- ++(0.00mm,0.00mm)
-- ++(0.00mm,0.00mm)
;
\node [black] at (84.85mm,-550mm) { $\lambda_1$ };
\node [black] at (884.89mm,-550mm) { $\lambda_2$ };
\node [black] at (1478.86mm,-550mm) { $\lambda_3$ };
\node [black] at (880mm,155.35mm) { $\alpha$ };
\node [black] at (549.52mm,70.49mm) { $e_1$ };
\node [black] at (900.07mm,454.35mm) { $e_2$ };
\node [black] at (1531.39mm,870.53mm) { $\mu$ };
\end{scope}
\end{tikzpicture} 
\right\} 
\quad \longleftrightarrow \quad 
\left\{ 
\begin{tikzpicture}[yscale=-1,xscale=-1,scale=0.03,baseline={([yshift=-.5ex]current bounding box.center)}]
\begin{scope}[shift={(0.00mm,719.29mm)}]
% path id='path4136'
% path spec='M 125.71429,-467.63778 1457.1429,838.07656'
\draw [fill=none,draw=black] (125.71mm,-467.64mm)
-- (1457.14mm,838.08mm)
;
% path id='path4138'
% path spec='m 1048.5714,435.21936 462.8572,-905.71428 0,0'
\draw [fill=none,draw=black] (1048.57mm,435.22mm)
-- ++(462.86mm,-905.71mm)
-- ++(0.00mm,0.00mm)
;
% path id='path4140'
% path spec='m 625.71429,20.93365 308.57142,-500 0,0 0,0'
\draw [fill=none,draw=black] (625.71mm,20.93mm)
-- ++(308.57mm,-500.00mm)
-- ++(0.00mm,0.00mm)
-- ++(0.00mm,0.00mm)
;
\node [black] at (84.85mm,-550mm) { $\lambda_1$ };
\node [black] at (884.89mm,-550mm) { $\lambda_2$ };
\node [black] at (1478.86mm,-550mm) { $\lambda_3$ };
\node [black] at (880mm,155.35mm) { $\beta$ };
\node [black] at (549.52mm,70.49mm) { $f_1$ };
\node [black] at (900.07mm,454.35mm) { $f_2$ };
\node [black] at (1531.39mm,870.53mm) { $\mu$ };
\end{scope}
\end{tikzpicture}
\right\}
\]

\noindent The $6j$ symbols are the entries in the change of basis matrix $(j_{\mu}^{\lambda_1,\lambda_2,\lambda_3})^{e_1,e_2}_{f_1,f_2}$. In other words, they are a coordinate representation of the associator. They must satisfy some algebraic relations which correspond to the pentagon axiom and the unit axiom. From the $6j$-symbols and the Grothendieck ring, you can recover the tensor category. Therefore, the $6j$-symbols are coordinates on the moduli stack of semi-simple tensor categories with a fixed Grothendieck ring.
\end{Definition}

\section{Young Semi-normal form}

\begin{Definition} \label{sec:young_seminormal_form}
Let $\XX$ be a semi-simple tensor category with distinguished object $X$. If $\sigma \in \End(X^{\otimes 2})$ then we have
\[
\begin{tikzpicture}[yscale=-1,scale=0.02,baseline={([yshift=-.5ex]current bounding box.center)}]
\begin{scope}[shift={(0.00mm,719.29mm)}]
% path id='path4155'
% path spec='m 937.2291,-396.38924 0,170.09083 553.9485,0 0,-168.69664 z'
\draw [fill=none,draw=black] (937.23mm,-396.39mm)
-- ++(0.00mm,170.09mm)
-- ++(553.95mm,0.00mm)
-- ++(0.00mm,-168.70mm)
-- cycle
;
% path id='path4157'
% path spec='M 88.893424,-713.38442 1349.5638,1052.3623'
\draw [fill=none,draw=black] (88.89mm,-713.38mm)
-- (1349.56mm,1052.36mm)
;
% path id='path4159'
% path spec='M 1054.5993,-224.47059 711.14739,158.13738'
\draw [fill=none,draw=black] (1054.60mm,-224.47mm)
-- (711.15mm,158.14mm)
;
% path id='path4161'
% path spec='M 1390,-226.2092 857.14286,365.21936'
\draw [fill=none,draw=black] (1390.00mm,-226.21mm)
-- (857.14mm,365.22mm)
;
% path id='path4165'
% path spec='m 1057.1429,-397.63778 -1.4286,-322.85714'
\draw [fill=none,draw=black] (1057.14mm,-397.64mm)
-- ++(-1.43mm,-322.86mm)
;
% path id='path4167'
% path spec='m 1401.4286,-396.2092 4.2857,-322.85715 0,1.42857'
\draw [fill=none,draw=black] (1401.43mm,-396.21mm)
-- ++(4.29mm,-322.86mm)
-- ++(0.00mm,1.43mm)
;
\node [black] at (100mm,-850mm) { $\lambda$ };
\node [black] at (450mm,100mm) { $a$ };
\node [black] at (800mm,600mm) { $b$ };
\node [black] at (1315.22mm,1200mm) { $\mu$ };
\node [black] at (1200.43mm,-307.30mm) { $\sigma$ };
\end{scope}
\end{tikzpicture} \quad = \sum_{f,g} m_{fg,ab}(\sigma)
\begin{tikzpicture}[yscale=-1,scale=0.02,baseline={([yshift=-.5ex]current bounding box.center)}]
\begin{scope}[shift={(0.00mm,719.29mm)}]
% path id='path4157'
% path spec='M 88.893424,-713.38442 1349.5638,1052.3623'
\draw [fill=none,draw=black] (88.89mm,-713.38mm)
-- (1349.56mm,1052.36mm)
;
% path id='path4159'
% path spec='M 1315.2186,-717.42497 711.14739,158.13738'
\draw [fill=none,draw=black] (1315.22mm,-717.42mm)
-- (711.15mm,158.14mm)
;
% path id='path4161'
% path spec='M 1638.4975,-717.14328 857.14286,365.21936'
\draw [fill=none,draw=black] (1638.50mm,-717.14mm)
-- (857.14mm,365.22mm)
;

\node [black] at (100mm,-850mm) { $\lambda$ };
\node [black] at (450mm,100mm) { $f$ };
\node [black] at (800mm,600mm) { $g$ };
\node [black] at (1315.22mm,1200mm) { $\mu$ };
\end{scope}
\end{tikzpicture}
\]
We call the matrix $m(\sigma)$ a {\bf semi-normal form} for $\sigma$.
\end{Definition}

\begin{Definition} \label{def:symmetric_group_tensor_cat}
We define the category $S$ which has objects the natural numbers and morphisms
\[
S(m,n) = 
\begin{cases}
S_m & m = n \\
0 & {\rm otherwise}.
\end{cases}
\]
where $S_n$ is the symmetric group with simple reflections $g_1,\dots,g_{n-1}$. The inclusion $S_m \otimes S_n \to S_{m+n}$ defined by $g_i \otimes g_j \mapsto g_i g_{m+j}$ equips $S$ with a tensor structure. We define $\mathcal{S} \subseteq [S^{\rm op},\Vect]$ to be the idempotent completion of $S$. The monoidal structure on $S$ extends to $\mathcal{S}$ via day convolution. The category $\mathcal{S}$ can be described as the polynomial representations of $\GL(\infty)$ as defined by Sam and Snowden in \cite{MR3430359}. The Grothendieck ring for $\mathcal{S}$ has basis given by partitions and multiplication given by the Littlewood-Richardson rule. A special case of the Littlewood-Richardson rule is the Pieri rule:
\[
\lambda \otimes \ydiagram{1} = \sum_{\lambda \subset \mu \vdash n+1} \mu
\]
This implies that the tree basis vectors
\[
\begin{tikzpicture}[yscale=-1,scale=0.03,baseline={([yshift=-.5ex]current bounding box.center)}]
\begin{scope}[shift={(0.00mm,719.29mm)}]
\draw [fill=black,draw=black] (659.15mm,-50.45mm) circle (4.80mm) ;
\draw [fill=black,draw=black] (743.15mm,45.55mm) circle (4.80mm) ;
\draw [fill=black,draw=black] (832.15mm,156.55mm) circle (4.80mm) ;
% path id='path4167'
% path spec='m 270.72496,-589.74787 -164.04877,166.8772'
\draw [fill=none,draw=black] (270.72mm,-589.75mm)
-- ++(-164.05mm,166.88mm)
;
% path id='path4171'
% path spec='m 502,-591.63777 -318,328'
\draw [fill=none,draw=black] (502.00mm,-591.64mm)
-- ++(-318.00mm,328.00mm)
;
% path id='path4173'
% path spec='M 16.656872,-581.15247 846.26401,1037.6598'
\draw [fill=none,draw=black] (16.66mm,-581.15mm)
-- (846.26mm,1037.66mm)
;
% path id='path4175'
% path spec='m 262.23968,-87.49391 497.80318,-492.14632 0,0 0,0'
\draw [fill=none,draw=black] (262.24mm,-87.49mm)
-- ++(497.80mm,-492.15mm)
-- ++(0.00mm,0.00mm)
-- ++(0.00mm,0.00mm)
;
% path id='path4177'
% path spec='m 341.43564,62.41273 692.96466,-650.53824 0,0'
\draw [fill=none,draw=black] (341.44mm,62.41mm)
-- ++(692.96mm,-650.54mm)
-- ++(0.00mm,0.00mm)
;
% path id='path4179'
% path spec='M 632.76363,613.95606 1749.9923,-559.84123'
\draw [fill=none,draw=black] (632.76mm,613.96mm)
-- (1749.99mm,-559.84mm)
;
\node [black] at (-50mm,-395.64mm) { $e_{m+1}$ };
\node [black] at (0.00mm,-247.64mm) { $e_{m+2}$ };
\node [black] at (50.00mm,-71.64mm) { $e_{m+3}$ };
\node [black] at (140.00mm,96.36mm) { $e_{m+4}$ };
\node [black] at (500.00mm,628.36mm) { $e_n$ };
\node [black] at (12.00mm,-660mm) { $\lambda$ };
\node [black] at (272.00mm,-660.64mm) { $\ydiagram{1}$ };
\node [black] at (488.00mm,-660.64mm) { $\ydiagram{1}$ };
\node [black] at (736.00mm,-660.64mm) { $\ydiagram{1}$ };
\node [black] at (1024.00mm,-660.64mm) { $\ydiagram{1}$ };
\node [black] at (1736.00mm,-660.64mm) { $\ydiagram{1}$ };
\node [black] at (848.00mm,1100mm) { $\mu$ };
\end{scope}
\end{tikzpicture}
\]
are in bijection (upto scaling) with standard skew tableaux of shape $\mu \backslash \lambda$. We shall abuse notation and identify these tree basis vectors with the corresponding standard skew tableaux. Suppose that $\lambda \subseteq \mu \vdash n+2$ are partitions such that $\mu \backslash \lambda$ is not contained in a single row or column. Then there are exactly two partitions which satisfy $\lambda \subseteq \nu \subseteq \mu$. Call them $\nu$ and $\nu'$. The multiplicity space $\mathcal{S}(\lambda \otimes \ydiagram{1} \otimes \ydiagram{1},\mu)$ is 2-dimensional with basis
\[
\begin{tikzpicture}[yscale=-1,scale=0.015,baseline={([yshift=-.5ex]current bounding box.center)}]
\begin{scope}[shift={(0.00mm,719.29mm)}]
% path id='path3355'
% path spec='M 125.71429,-716.2092 1522.8571,1058.0766'
\draw [fill=none,draw=black] (125.71mm,-716.21mm)
-- (1522.86mm,1058.08mm)
;
% path id='path3359'
% path spec='m 880,-716.2092 -402.85714,445.71428 0,0'
\draw [fill=none,draw=black] (880.00mm,-716.21mm)
-- ++(-402.86mm,445.71mm)
-- ++(0.00mm,0.00mm)
;
% path id='path3361'
% path spec='M 891.42857,258.07651 1708.5714,-713.35206'
\draw [fill=none,draw=black] (891.43mm,258.08mm)
-- (1708.57mm,-713.35mm)
;
\node [black] at (200mm,-461.92mm) { $\lambda$ };
\node [black] at (450mm,100mm) { $\nu$ };
\node [black] at (900mm,700mm) { $\mu$ };
\end{scope}
\end{tikzpicture} \; , \;
\begin{tikzpicture}[yscale=-1,scale=0.015,baseline={([yshift=-.5ex]current bounding box.center)}]
\begin{scope}[shift={(0.00mm,719.29mm)}]
% path id='path3355'
% path spec='M 125.71429,-716.2092 1522.8571,1058.0766'
\draw [fill=none,draw=black] (125.71mm,-716.21mm)
-- (1522.86mm,1058.08mm)
;
% path id='path3359'
% path spec='m 880,-716.2092 -402.85714,445.71428 0,0'
\draw [fill=none,draw=black] (880.00mm,-716.21mm)
-- ++(-402.86mm,445.71mm)
-- ++(0.00mm,0.00mm)
;
% path id='path3361'
% path spec='M 891.42857,258.07651 1708.5714,-713.35206'
\draw [fill=none,draw=black] (891.43mm,258.08mm)
-- (1708.57mm,-713.35mm)
;
\node [black] at (200mm,-461.92mm) { $\lambda$ };
\node [black] at (450mm,100mm) { $\nu'$};
\node [black] at (1000mm,700mm) { $\mu$ };
\end{scope}
\end{tikzpicture}
\]
The semi-normal form for $g_1$ is well known:
\[
m(g_1) = 
\begin{pmatrix} -1/k & 1 \\ 1 - 1/k^2 & 1/k \end{pmatrix}
\]
where $k$ is the axial distance in $\mu \backslash \lambda$:
\[
\begin{tikzpicture}[yscale=-1,scale=0.03,baseline={([yshift=-.5ex]current bounding box.center)}]
\begin{scope}[shift={(0.00mm,719.29mm)}]
% path id='path3338'
% path spec='m 162.85714,-399.06635 0,1145.71431 302.85715,0 0,-334.2857 360,0 0,-348.57146 382.85711,0 5.7143,-454.28572'
\draw [fill=none,draw=black,dashed] (162.86mm,-399.07mm)
-- ++(0.00mm,1145.71mm)
-- ++(302.86mm,0.00mm)
-- ++(0.00mm,-334.29mm)
-- ++(360.00mm,0.00mm)
-- ++(0.00mm,-348.57mm)
-- ++(382.86mm,0.00mm)
-- ++(5.71mm,-454.29mm)
;
% path id='path3340'
% path spec='m 162.85714,-396.2092 1051.42856,0 1e-4,6.33928'
\draw [fill=none,draw=black,dashed] (162.86mm,-396.21mm)
-- ++(1051.43mm,0.00mm)
-- ++(0.00mm,6.34mm)
;
% path id='path4161'
% path spec='M 824.28572,63.07651 940,63.79079 l 0,110.71429 -115,0 z'
\draw [fill=none,draw=black] (824.29mm,63.08mm)
-- (940.00mm,63.79mm)
-- ++(0.00mm,110.71mm)
-- ++(-115.00mm,0.00mm)
-- cycle
;
% path id='path4163'
% path spec='m 462.65115,412.44711 115.71428,0.71428 0,110.71429 -115,0 z'
\draw [fill=none,draw=black] (462.65mm,412.45mm)
-- ++(115.71mm,0.71mm)
-- ++(0.00mm,110.71mm)
-- ++(-115.00mm,0.00mm)
-- cycle
;
% path id='path4167'
% path spec='M 518.90014,412.43922 517.709,119.20535 825.48283,118.9091'
\draw [fill=none,draw=black] (518.90mm,412.44mm)
-- (517.71mm,119.21mm)
-- (825.48mm,118.91mm)
;
\end{scope}
\end{tikzpicture}
\]
\end{Definition}
Identify the trivalent vertex $e$ with the matrix unit
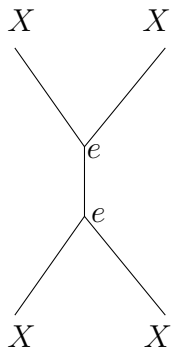
\[
\begin{tikzpicture}[yscale=-1,scale=0.02,baseline={([yshift=-.5ex]current bounding box.center)}]
\begin{scope}[shift={(0.00mm,719.29mm)}]
% path id='path4155'
% path spec='M 277.14286,-716.2092 740,-59.06635 1277.1429,-716.2092'
\draw [fill=none,draw=black] (277.14mm,-716.21mm)
-- (740.00mm,-59.07mm)
-- (1277.14mm,-716.21mm)
;
% path id='path4157'
% path spec='m 740,-58.70921 0,99.64286 0,0'
\draw [fill=none,draw=black] (740.00mm,-58.71mm)
-- ++(0.00mm,99.64mm)
-- ++(0.00mm,0.00mm)
;
% path id='path4159'
% path spec='M 277.14286,1061.1172 740,403.9744 l 537.1429,657.1428'
\draw [fill=none,draw=black] (277.14mm,1061.12mm)
-- (740.00mm,403.97mm)
-- ++(537.14mm,657.14mm)
;
% path id='path4161'
% path spec='M 740,403.61726 740,37.29415'
\draw [fill=none,draw=black] (740.00mm,403.62mm)
-- (740.00mm,37.29mm)
;
\node [black] at (802.86mm,-27.64mm) { $e$ };
\node [black] at (831.43mm,398.08mm) { $e$ };
\node [black] at (322.86mm,-900mm) { $X$ };
\node [black] at (1220.00mm,-900mm) { $X$ };
\node [black] at (322.86mm,1200mm) { $X$ };
\node [black] at (1228.57mm,1200mm) { $X$ };
\end{scope}
\end{tikzpicture}
\]
inside $\End(X^{\otimes 2})$. Consider the action of $e$ on the two bases
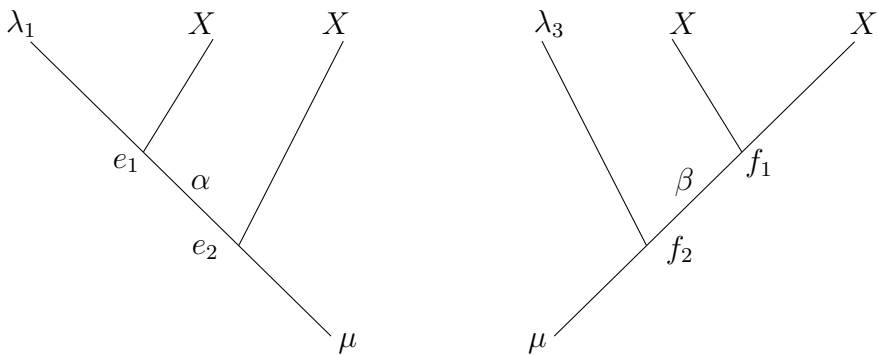
\[
\begin{tikzpicture}[yscale=-1,scale=0.03,baseline={([yshift=-.5ex]current bounding box.center)}]
\begin{scope}[shift={(0.00mm,719.29mm)}]
% path id='path4136'
% path spec='M 125.71429,-467.63778 1457.1429,838.07656'
\draw [fill=none,draw=black] (125.71mm,-467.64mm)
-- (1457.14mm,838.08mm)
;
% path id='path4138'
% path spec='m 1048.5714,435.21936 462.8572,-905.71428 0,0'
\draw [fill=none,draw=black] (1048.57mm,435.22mm)
-- ++(462.86mm,-905.71mm)
-- ++(0.00mm,0.00mm)
;
% path id='path4140'
% path spec='m 625.71429,20.93365 308.57142,-500 0,0 0,0'
\draw [fill=none,draw=black] (625.71mm,20.93mm)
-- ++(308.57mm,-500.00mm)
-- ++(0.00mm,0.00mm)
-- ++(0.00mm,0.00mm)
;
\node [black] at (84.85mm,-550mm) { $\lambda_1$ };
\node [black] at (884.89mm,-550mm) { $X$ };
\node [black] at (1478.86mm,-550mm) { $X$ };
\node [black] at (880mm,155.35mm) { $\alpha$ };
\node [black] at (549.52mm,70.49mm) { $e_1$ };
\node [black] at (900.07mm,454.35mm) { $e_2$ };
\node [black] at (1531.39mm,870.53mm) { $\mu$ };
\end{scope}
\end{tikzpicture} \hspace{2cm}
\begin{tikzpicture}[yscale=-1,xscale=-1,scale=0.03,baseline={([yshift=-.5ex]current bounding box.center)}]
\begin{scope}[shift={(0.00mm,719.29mm)}]
% path id='path4136'
% path spec='M 125.71429,-467.63778 1457.1429,838.07656'
\draw [fill=none,draw=black] (125.71mm,-467.64mm)
-- (1457.14mm,838.08mm)
;
% path id='path4138'
% path spec='m 1048.5714,435.21936 462.8572,-905.71428 0,0'
\draw [fill=none,draw=black] (1048.57mm,435.22mm)
-- ++(462.86mm,-905.71mm)
-- ++(0.00mm,0.00mm)
;
% path id='path4140'
% path spec='m 625.71429,20.93365 308.57142,-500 0,0 0,0'
\draw [fill=none,draw=black] (625.71mm,20.93mm)
-- ++(308.57mm,-500.00mm)
-- ++(0.00mm,0.00mm)
-- ++(0.00mm,0.00mm)
;
\node [black] at (84.85mm,-550mm) { $X$ };
\node [black] at (884.89mm,-550mm) { $X$ };
\node [black] at (1478.86mm,-550mm) { $\lambda_3$ };
\node [black] at (880mm,155.35mm) { $\beta$ };
\node [black] at (549.52mm,70.49mm) { $f_1$ };
\node [black] at (900.07mm,454.35mm) { $f_2$ };
\node [black] at (1531.39mm,870.53mm) { $\mu$ };
\end{scope}
\end{tikzpicture}
\]
On the left basis, $e$ acts via the semi-normal form $m(e)$. On the right basis, $e$ is a projection. Inside $\mathcal{S}$, this implies the matrix of eigenvectors vectors of $m(\ydiagram{2})= 1/2(e + m(g_1))$ equals $(j_{\mu}^{\lambda,\ydiagram{1},\ydiagram{1}})^{-1}$. Therefore we have
\[
(j_{\mu}^{\lambda,\ydiagram{1},\ydiagram{1}})^{-1} = 
\begin{pmatrix}
\frac{k}{k+1} & \frac{-k}{k-1} \\
1 & 1
\end{pmatrix}.
\]
This proves theorem \ref{thm:main_theorem}.
\bibliographystyle{alpha}
\bibliography{bibliography}

\end{document}